\documentclass{tran-l}

\usepackage{epsfig}  		
\usepackage{epic,eepic}
\usepackage{amsfonts}
\usepackage{amsmath}
\usepackage{amssymb}
\usepackage{wasysym}
\usepackage{yfonts}  
\usepackage{amsthm}
\usepackage{amscd}   
\usepackage[all]{xypic}
\usepackage{mathrsfs}

\CompileMatrices

\newtheorem{theorem}{Theorem}[section]
\newtheorem{lemma}[theorem]{Lemma}
\newtheorem{proposition}[theorem]{Proposition}
\newtheorem{corollary}[theorem]{Corollary} 
\newtheorem*{definition}{Definition}
\newtheorem{example}[theorem]{Example}
\newtheorem*{remark}{Remark}
\newtheorem*{notation}{Notation and terminology}
\newtheorem*{acknowledgement}{Acknowledgement}
\newtheorem*{thmA}{Theorem A}
\newtheorem*{thmB}{Theorem B}
\newtheorem*{thmC}{Theorem C}
\newtheorem*{thmA'}{Theorem A*}
\newtheorem*{propA}{Proposition A}

\numberwithin{equation}{section}

\begin{document}

\title[Cohomology of solvable groups]{The cohomology of virtually torsion-free solvable groups of finite rank}
\author[Kropholler]{Peter Kropholler}
\address{
Mathematical Sciences\\
University of Southampton\\
Highfield\\
Southampton SO17 1BJ\\
UK}
\email{P.H.Kropholler@southampton.ac.uk}

\author[Lorensen]{Karl Lorensen}
\address{Department of Mathematics and Statistics\\
Pennsylvania State University, Altoona College\\
Altoona, PA 16601\\
USA}
\email{kql3@psu.edu}


\begin{abstract}
 Assume that $G$ is a virtually torsion-free solvable group of finite rank and $A$ a $\mathbb ZG$-module whose underlying abelian group is torsion-free and has finite rank. We stipulate a condition on $A$ that ensures that $H^n(G,A)$ and $H_n(G,A)$ are finite for all $n\geq 0$. Using this property for cohomology in dimension two, we deduce two results concerning the presence of near supplements and complements in solvable groups of finite rank.  As an application of our near-supplement theorem, we obtain a new result regarding the homological dimension of solvable groups.   

\end{abstract}

\maketitle

\section{Introduction}

The primary goal of this paper is to establish a property of the cohomology and homology of a virtually torsion-free solvable group of finite rank. We say that a solvable group has {\it finite rank} if, for every prime $p$, each of its abelian $p$-sections is a direct sum of finitely many cyclic and quasicyclic $p$-groups. In the literature, this property is commonly referred to as having {\it finite abelian section rank}. For our result, we will require the concept of the spectrum of a solvable group of finite rank, a notion that is usually only associated with the special case of a solvable minimax group. If $G$ is a solvable group of finite rank, then the {\it spectrum} of $G$, denoted ${\rm spec}(G)$, is the set of primes $p$ for which $G$ has a quasicyclic $p$-section. Given an arbitrary set of primes $\pi$, we describe a solvable group of finite rank as {\it $\pi$-spectral} if its spectrum is contained in $\pi$. Equipped with these terms, we can state our main theorem.

\begin{thmA} Let $G$ be a virtually torsion-free solvable group of finite rank with spectrum $\pi$. Assume that $A$ is a $\mathbb ZG$-module whose underlying abelian group is torsion-free and has finite rank. Suppose further that $A$ does not have any nontrivial $\mathbb ZG$-submodules that are $\pi$-spectral as abelian groups. Then $H^n(G,A)$ and $H_n(G,A)$ are finite
for $n\geq 0$. 

\end{thmA}

Section 2 is devoted to the proof of this theorem. We begin the section by establishing the following property of the functors ${\rm Ext}_{\mathbb ZG}^n$ and  ${\rm Tor}^{\mathbb ZG}_n$ for $G$ an abelian group. As well as playing a pivotal role in the proof of Theorem A, this proposition may be of  interest in its own right. 

\begin{propA} Let $G$ be an abelian group. Assume that $A$ and $B$ are $\mathbb ZG$-modules whose additive groups are torsion-free abelian groups of finite rank.  Suppose further that $A$ fails to contain a nontrivial submodule that is isomorphic to a submodule of $B$. Then there is a positive integer $m$ such that  $$m\cdot {\rm Ext}_{\mathbb ZG}^n(A,B)=0\ \ {\rm and}\ \ m\cdot {\rm Tor}^{\mathbb ZG}_n(A,B)=0$$ 
for all $n\geq 0$.  
\end{propA}

In Sections 3 and 4,  we apply the cohomological part of Theorem A in dimension two to shed light on the structure of solvable groups of finite rank. Our interest is in discerning the presence of certain types of near supplements and complements.  A {\it near supplement} to a normal subgroup $K$ of a group $G$ is a subgroup $X$ such that $[G:KX]$ is finite; if, in addition, $K\cap X=1$, then $X$ is referred to as a {\it near complement} to $K$.  The relevance of Theorem A to the detection of near supplements and complements is revealed by the following result of D. J. S. Robinson, applied in  [{\bf 7}, Theorem C] to find nilpotent near supplements. 

 \begin{proposition}{\rm (Robinson [{\bf 8}, Theorem 10.1.15])} Assume that $G$ is a group and $A$ a $\mathbb ZG$-module that has finite rank as an abelian group. Let $\xi$ be an element of $H^2(G,A)$ and $1\rightarrow A\rightarrow E \rightarrow G\rightarrow 1$ a group extension corresponding to $\xi$. Then $\xi$ has finite order if and only if $E$ contains a subgroup $X$ such that $X\cap A$ is finite and $[E:AX]$ is finite. 
\end{proposition}

 Our focus is on near supplements and complements to normal subgroups that give rise to a quotient that is minimax. We remind the reader that a solvable group is {\it minimax} if it has a series of finite length in which each factor is either cyclic or quasicyclic. Moreover, if $\pi$ is a set of primes, then a {\it $\pi$-minimax group} is a solvable minimax group that is $\pi$-spectral.  Given a solvable group $G$ with finite rank and $K\lhd G$ such that $G/K$ is $\pi$-minimax for a set of primes $\pi$, we seek conditions that guarantee the existence of a $\pi$-minimax near supplement or complement to $K$. 

Section 3 treats near supplements; our principal discovery is 

\begin{thmB} Let $G$ be a solvable group with finite rank. Assume that $K$ is a normal subgroup of $G$ such that $G/K$ is $\pi$-minimax for some set $\pi$ of primes. If $G/K$ is virtually torsion-free, then $K$ has a $\pi$-minimax near supplement in $G$.
\end{thmB}

\noindent The two key ingredients for Theorem B are Theorem A and a result of D. J. S. Robinson \cite{robinson4} that describes a situation where the cohomology of a solvable minimax group with torsion coefficients must itself be torsion. With these cohomological tools, the proof of Theorem B becomes a rather swift affair, involving induction on the Hirsch length of $K$.

Theorem B is a generalization of [{\bf 4}, Proposition I.20], which states that, if $G$ is a solvable minimax group and $K\lhd G$ such that $G/K$ is polycyclic, then $K$ has a polycyclic near supplement. The latter proposition is employed by M. R. Bridson and the first author in order to prove a result [{\bf 4}, Theorem I.3] about the homological dimension of an abelian-by-polycyclic group. We conclude Section 3 by applying Theorem B to generalize Bridson and the first author's theorem  to a larger class of groups (Proposition 3.6). 

In the final section of the paper, we investigate which hypotheses need to be added to Theorem B in order to give rise to a near complement to $K$. It turns out that near complements are a comparatively rare phenomenon, their presence assured only if $K$ happens to enjoy two quite restrictive properties. The first condition is that $K$ should be Noetherian as a $G$-operator group, meaning that it satisfies the maximum condition on $G$-invariant subgroups. Second, $K$ must not have any torsion-free $\pi$-minimax quotients.  That these two hypotheses suffice to furnish a near complement forms the content of Theorem C below. The proof of this result hinges on recognizing that the second property is often passed on to $G$-invariant subgroups of Noetherian $G$-operator groups.  

Before proceeding to the body of the paper, we relate the terms and notation that we will be using. In addition, we state a proposition describing the basic structure of solvable groups of finite rank. 

\begin{notation}
{\rm Throughout the paper, $\pi$ will represent an arbitrary set of primes. 

A {\it section} of a group or module is a quotient of one of its subgroups or submodules, respectively. 

An abelian group $A$ is said to be {\it bounded} if there is a positive integer $m$ such that $mA=0$. 
 
A {\it \v{C}ernikov group} is a group that is a finite extension of a direct product of finitely many quasicyclic groups.

Let $G$ be a group. Then $\tau(G)$ represents the {\it torsion radical} of $G$, that is, the join of all the torsion normal subgroups of $G$.  The join of all the nilpotent normal subgroups of $G$, known as the {\it Fitting subgroup}, is denoted by ${\rm Fitt}(G)$. 

If $G$ is a solvable group such that $G/\tau(G)$ has finite rank, then we say that $G$ has {\it finite torsion-free rank}. 

If $G$ is a solvable group of finite torsion-free rank, then $h(G)$ denotes the {\it Hirsch length} of $G$, namely, the number of infinite cyclic factors in any series of finite length whose factors are all either infinite cyclic or torsion. 
If $G$ is also minimax, then the {\it minimax length} of $G$, written $m(G)$, is the number of infinite factors in any series of finite length in which each factor is either cyclic or quasicyclic. Like the Hirsch length, the minimax length is an invariant of the group; in other words, it does not depend on the particular series selected. 

Let $G$ be a group and $A$ a $\mathbb ZG$-module. If $A\otimes \mathbb Q$ is a simple $\mathbb QG$-module, then we will refer to $A$ as {\it rationally irreducible}. In other words, $A$ is rationally irreducible if and only if the additive group of $A$ is not torsion and, for every submodule $B$ of $A$, either $B$ or $A/B$ is torsion as an abelian group.
}
\end{notation}

The following properties will be invoked frequently; they can be found in [{\bf 8}, Section 5.2]. 

\begin{proposition} Assume that $G$ is a solvable group of finite rank. 
\vspace{2pt}

\noindent (i) $G$ is virtually torsion-free if and only if $\tau(G)$ is finite.
\vspace{2pt}

\noindent (ii) If $G$ is virtually torsion-free, then $G$ has a characteristic subgroup of finite index that possesses a characteristic series of finite length whose factors are torsion-free abelian groups.
\vspace{2pt}

\noindent (iii) If $\tau(G)$ is  a \v{C}ernikov group, then ${\rm Fitt(G)}$ is nilpotent and $G/{\rm Fitt}(G)$ virtually abelian. 
Moreover, if $\tau(G)$ is finite, then $G/{\rm Fitt}(G)$ is also finitely generated. 

\end{proposition}
\newpage

\section{Cohomology and homology}

The aim of this section is to prove Theorem A.  Our first step is to establish Proposition~A. For this purpose, we require the following property, which can be proved in the same fashion as Proposition 1.1.

\begin{proposition} Let $R$ be a ring. Suppose that $A$ and $B$ are $R$-modules such that the additive group of $B$ has finite rank.  Let $0\rightarrow B\rightarrow E\rightarrow A\rightarrow 0$  be an $R$-module extension and $\xi$ the element of ${\rm Ext}_R^1(A,B)$ that corresponds to this extension. Then $\xi$ has finite order if and only if $E$ has a submodule $X$ such that $B\cap X$ and $E/B+X$ are finite. 
\end{proposition}

We now proceed with the proof of Proposition~A; the argument we give for the case where both $A$ and $B$ are rationally irreducible is due to Derek Robinson.  

\begin{proof}[{\bf Proof of Proposition A}]

First we treat the case where $A$ and $B$ are both rationally irreducible. Set $R=\mathbb ZG$. Let $a$ and $b$ be nonzero elements of $A$ and $B$, respectively. Assume that  $I={\rm Ann}_R(a)$ and $J={\rm Ann}_R(b)$.  Then $R/I\cong Ra$ and $R/J\cong Rb$ as $R$-modules. We claim that $IA=0$ and $JB=0$. To establish the former assertion, let $x\in A$. Since $A/Ra$ is torsion {\it qua} abelian group, we have $mx=ra$ for some positive integer $m$ and $r\in R$. Thus $m(Ix)=r(Ia)=0$, so that $Ix=0$. Hence $IA=0$, and, by the same reasoning, $JB=0$.  

Next we show that $I\not \subset J$. Suppose that $I\subset J$.  Because $R/I$ is a rationally irreducible $R$-module, it follows that $I=J$. However, this implies that ${Ra\cong Rb}$ as $R$-modules, contradicting the hypothesis. Therefore, $I\not \subset J$. The rational irreducibility of $R/J$, then, forces the additive group of $R/I+J$ to be torsion. As a result, there is a positive integer $m$ such that $m\cdot 1_R\in I+J$. Now we observe that, since $R$ is commutative,  ${\rm Ext}_R^n(A,B)$ and ${\rm Tor}^R_n(A,B)$ can be endowed with $R$-module structures. Moreover, as $R$-modules, both
${\rm Ext}_R^n(A,B)$ and ${\rm Tor}^R_n(A,B)$ are annihilated by $I+J$. Consequently, $m{\rm Ext}_R^n(A,B)=0$ and $m{\rm Tor}^R_n(A,B)=0$. 

The second case that we consider is the one where merely $B$ is rationally irreducible. We proceed here by induction on $h(A)$, the case $h(A)=0$ being trivial. Assume that $h(A)\geq 1$, and let $U$ be a rationally irreducible submodule of $A$ such that the additive group of $A/U$ is torsion-free. The conclusion of the proposition will follow from the case proved above and the inductive hypothesis if we succeed in showing that $A/U$ does not contain a nontrivial submodule that is isomorphic to a submodule of $B$. To accomplish this, suppose that $A/U$ possesses such a submodule and take $V$ to be a submodule of $A$ containing $U$ such that $V/U\cong W$, where $W$ is a nontrivial submodule of $B$. By the case of the proposition established above, ${\rm Ext}^1_R(V/U,U)$ is bounded. According to Proposition~2.1, this means that $V$ has a submodule that is isomorphic to a submodule of $W$ with finite index. But the presence of such a submodule contradicts our hypothesis. Therefore, we are compelled to conclude that $A/U$ fails to possess a nontrivial submodule that is isomorphic to a submodule of $B$. This completes the proof for the case where $B$ is rationally irreducible.

Finally, we handle the general case of Proposition A  by inducting on $h(B)$. Assume that $h(B)\geq 1$, and let $U$ be a rationally irreducible submodule of $B$ such that $B/U$ is torsion-free as an abelian group. In order to deduce the desired property from the case above and the inductive hypothesis, we only need to establish that $B/U$ fails to contain a nontrivial submodule that is isomorphic to a submodule of~$A$. This can be accomplished by adducing an argument virtually identical to the one in the previous paragraph. 

\end{proof}

Our proof of Theorem A will employ Proposition A in conjunction with the following homological isomorphisms. Although these are doubtless well known, we provide proofs for the sake of completeness.  

\begin{lemma} Let $G$ be a group and $R$ a commutative ring. Suppose that $A$ and $B $ are $RG$-modules and view both ${\rm Hom}_R(A,B)$ and $A\otimes_R B$ as $\mathbb ZG$-modules under the diagonal actions. 

\noindent (i) If either $A$ is projective or $B$ is injective as an $R$-module, then, for $n\geq 0$,

$${\rm Ext}^n_{RG}(A,B)\cong H^n(G,{\rm Hom}_R(A,B)).$$

\noindent (ii) If $B$ is flat as an $R$-module, then, for $n\geq 0$,

$${\rm Tor}_n^{RG}(A,B)\cong H_n(G, A\otimes_R B).$$

\end{lemma}

The isomorphisms in Lemma 2.2 can be gleaned from the spectral sequences  below. 

\begin{proposition} Let $G$ be a group and $R$ a commutative ring. Suppose that $A$ and $B $ are $RG$-modules and regard ${\rm Ext}^n_R(A,B)$ and ${\rm Tor}_n^R(A,B)$ as $\mathbb ZG$-modules via the diagonal action for $n\geq 0$. 
Then the following two statements hold. 
\vspace{5pt}

\noindent (i) There is a cohomology spectral sequence whose $E_2$-page is given by 
$$E_2^{pq}=H^p(G,{\rm Ext}^q_R(A,B)),$$
\noindent and that converges to ${\rm Ext}^n_{RG}(A,B)$.
\vspace{5pt}

\noindent (ii) There is a homology spectral sequence whose $E^2$-page is given by 
$$E^2_{pq}=H_p(G,{\rm Tor}_q^R(A,B)),$$
\noindent and that converges to ${\rm Tor}_n^{RG}(A,B)$.
\end{proposition}

\begin{proof} We will confine our attention to (i), as (ii) can be proved by a dual argument. The former statement will follow immediately from the Grothendieck spectral sequence in [{\bf 5}, Proposition VIII.9.3] provided that we verify that ${H^n(G, {\rm Hom}_R(A,I))=0}$ for any injective $RG$-module $I$ and $n\geq 1$. To accomplish this, take a projective $RG$-module resolution $\cdots \rightarrow P_1\rightarrow P_0\rightarrow A\rightarrow 0$ of $A$. Since $I$ is injective as an $RG$-module, it is also injective as an $R$-module. Thus the sequence

\begin{equation} 0\rightarrow {\rm Hom}_R(A,I)\rightarrow {\rm Hom}_R(P_0,I)\rightarrow {\rm Hom}_R(P_1,I)\rightarrow \cdots 
\end{equation}
is exact. We claim that $H^n(G, {\rm Hom}_R(P_i,I))=0$ for $i\geq 0$ and $n\geq 1$.
To show this, it suffices to establish that $H^n(G, {\rm Hom}_R(RG,I))=0$ for $n\geq 1$.
However, this is easily seen to be true because ${\rm Hom}_R(RG,I)$ is isomorphic to the coinduced module ${\rm Hom}(\mathbb ZG,I)$, which has trivial cohomology in every positive dimension. Therefore, (2.1) is an acyclic resolution of ${\rm Hom}_R(A,I)$ with respect to the functor $H^0(G,\ \_\ )$. As a consequence, the groups $H^n(G,{\rm Hom}_R(A,I))$ are the cohomology groups of the cochain complex
  $$0\rightarrow {\rm Hom}_{RG}(P_0,I)\rightarrow {\rm Hom}_{RG}(P_1,I)\rightarrow \cdots .$$
 But this complex is acyclic in view of the injectivity of $I$ as an $RG$-module. Hence  $H^n(G,{\rm Hom}_R(A,I))=0$ for $n\geq 1$.
  \end{proof}

In order to deduce Theorem A from Proposition A and Lemma 2.2, we will make use of the following class of modules. 

\begin{definition} {\rm Assume that $G$ is a group. Let $\mathfrak{C}(G, \pi)$ be the smallest class of $\mathbb ZG$-modules with the following two properties.
\vspace{5pt}

\noindent (i) The class $\mathfrak{C}(G, \pi)$ contains every $\mathbb ZG$-module whose underlying abelian group is torsion-free of finite rank and $\pi$-spectral. 
\vspace{5pt}

\noindent (ii) The class $\mathfrak{C}(G, \pi)$ is closed under forming $\mathbb ZG$-module quotients as well as extensions.}
\end{definition}

As an immediate consequence of the definition, we have that  $\mathfrak{C}(G, \pi)$ is also submodule-closed and therefore section-closed. This can be proved very easily by inducting on the number of closure operations from (ii) required to construct a module in $\mathfrak{C}(G, \pi)$; the details are left to the reader.

\begin{lemma} For any group $G$, the class $\mathfrak{C}(G, \pi)$ is closed under forming $\mathbb ZG$-module sections.
\end{lemma} 

Below we establish another closure property of $\mathfrak{C}(G, \pi)$. 

\begin{lemma} Assume that $G$ is a group. Suppose that $B$ is a $\mathbb ZG$-module whose additive group is torsion-free of finite rank and $\pi$-spectral.
If $A$ is a $\mathbb ZG$-module in  $\mathfrak{C}(G, \pi)$, then $A\otimes B$ lies in $\mathfrak{C}(G, \pi)$, where $A\otimes B$ is viewed as a $\mathbb ZG$-module under the diagonal action. 
\end{lemma}

\begin{proof} First we make the following three observations concerning $B$. 

\noindent (i) If $M$ is a $\mathbb ZG$-module whose additive group is torsion-free of finite rank and $\pi$-spectral, then $M\otimes B$ is torsion-free of finite rank and $\pi$-spectral. 

\noindent (ii) If  $M$ is a $\mathbb ZG$-module and $\bar{M}$ is a $\mathbb ZG$-module quotient of $M$, then $\bar{M}\otimes B$ is a $\mathbb ZG$-module quotient of $M\otimes B$. 

\noindent (iii) If $M$, $M'$, and $M''$ are $\mathbb ZG$-modules such that $M$ is an extension of $M'$ by $M''$, then $M\otimes B$ is a $\mathbb ZG$-module extension of $M'\otimes B$ by $M''\otimes B$. 

\noindent From these three properties it follows that $A\otimes B$ belongs to $\mathfrak{C}(G, \pi)$  by induction on the number of closure operations required to construct $A$ from $\mathbb ZG$-modules that are torsion-free of finite rank and $\pi$-spectral. Statement (i) establishes the base case, and (ii) and (iii) permit the execution of the inductive step. 
\end{proof}

Proposition A gives rise to the following properties of ${\rm Ext}^n_{\mathbb ZG}$ and ${\rm Tor}_n^{\mathbb ZG}$ applied to modules in $\mathfrak{C}(G, \pi)$ when $G$ is abelian. 

\begin{lemma} Let $G$ be an abelian group. Assume that $B$ is a $\mathbb ZG$-module whose additive group is torsion-free with finite rank. Suppose further that there are no nontrivial $\mathbb ZG$-submodules of $B$ that are $\pi$-spectral as abelian groups. If $A$ is a $\mathbb ZG$-module in $\mathfrak{C}(G, \pi)$ and $\bar{B}$ a $\mathbb ZG$-module quotient of $B$, then ${\rm Ext}_{\mathbb ZG}^n(A,\bar{B})$ and ${\rm Tor}^{\mathbb ZG}_n(A,\bar{B})$ are bounded abelian groups for $n\geq 0$. 
\end{lemma}

\begin{proof}  We will just prove the result for  ${\rm Ext}^n_{\mathbb ZG}$; the same reasoning will then apply to ${\rm Tor}_n^{\mathbb ZG}$. For each nonnegative integer $k$, let $\mathfrak{C}_k(G,\pi)$ be the class of all $\mathbb ZG$-modules that can be constructed from $\mathbb ZG$-modules that are torsion-free of finite rank and $\pi$-spectral as abelian groups by a sequence of at most $k$ homomorphic images and extensions. Taking $\bar{B}$ to be a $\mathbb ZG$-module quotient of $B$, we will prove by induction on $k$ that, for every module $A$ belonging to $\mathfrak{C}_k(G,\pi)$, ${\rm Ext}_{\mathbb ZG}^n(A,\bar{B})$ is a bounded abelian group for $n\geq 0$. The case $k=0$ follows from Proposition A and the long exact ${\rm Ext}_{\mathbb ZG}$-sequence. Suppose that $k>0$, and let $A$ be a module in the class $\mathfrak{C}_k(G,\pi)$.
We distinguish two possibilities: (1) $A$ is a $\mathbb ZG$-module homomorphic image of a module $M$ in $\mathfrak{C}_{k-1}(G,\pi)$; 
(2) $A$ is a $\mathbb ZG$-module extension of $A'$ by $A''$ such that $A'$ and $A''$ are both in $\mathfrak{C}_{k-1}(G,\pi)$. First we treat case (2). Under this assumption,  ${\rm Ext}_{\mathbb ZG}^n(A',\bar{B})$ and ${\rm Ext}_{\mathbb ZG}^n(A'',\bar{B})$ are bounded for $n\geq 0$, in view of the inductive hypothesis. Therefore,  ${\rm Ext}_{\mathbb ZG}^n(A,\bar{B})$ is bounded for $n\geq 0$. Assume now that case (1) holds. Let $M'$ be a submodule of $M$ such that $M/M'\cong A$. It is straightforward to see that every submodule of a module in $\mathfrak{C}_{k-1}(G,\pi)$ is also in $\mathfrak{C}_{k-1}(G,\pi)$.  Thus both ${\rm Ext}_{\mathbb ZG}^n(M,\bar{B})$ and ${\rm Ext}_{\mathbb ZG}^n(M',\bar{B})$ are bounded for $n\geq 0$. Once more availing ourselves of the long exact ${\rm Ext}_{\mathbb ZG}$-sequence,  we conclude that ${\rm Ext}_{\mathbb ZG}^n(A,\bar{B})$ is bounded for $n\geq 0$.
\end{proof}

Our purpose in defining the class  $\mathfrak{C}(G, \pi)$ and enunciating Lemma 2.6 is to apply the lemma to the integral homology of a nilpotent normal subgroup of a torsion-free solvable group with finite rank  and spectrum $\pi$ in place of the module $A$. To this end, we require the following property.

\begin{lemma} Assume that $G$ is a group, $N\lhd G$, and $Q=G/N$. In addition, suppose that $N$ is nilpotent of finite rank, $\pi$-spectral, and torsion-free. Then, for each $n\geq 0$, the $\mathbb ZQ$-module $H_nN$ lies in the class $\mathfrak{C}(Q, \pi)$.
\end{lemma}

\begin{proof} We proceed by induction on the nilpotency class of $N$. First suppose that $N$ is abelian. Then $H_nN$ is torsion-free for $n\geq 0$. Also, as an exterior power of $N$, $H_nN$ is $\pi$-spectral.  Hence $H_nN$ belongs to $\mathfrak{C}(Q, \pi)$.
Now assume that the nilpotency class of $N$ exceeds one.  Set $Z=Z(N)$ and consider the Lyndon-Hochschild-Serre (LHS) homology spectral sequence associated to the extension $1\rightarrow Z\rightarrow N\rightarrow N/Z\rightarrow 1$. In this spectral sequence, $E_{pq}^2=H_p(N/Z, H_qZ)=H_p(N/Z)\otimes H_qZ$ for $p,\ q \geq 0$. By the inductive hypothesis, $H_p(N/Z)$ lies in $\mathfrak{C}(Q, \pi)$. Therefore, by Lemma 2.5, $E_{pq}^2$ belongs to $\mathfrak{C}(Q, \pi)$. Since the differentials in the spectral sequence are $\mathbb ZQ$-module homomorphisms, $E_{pq}^{\infty}$ can be regarded as a $\mathbb ZQ$-module section of $E_{pq}^2$. Thus $E_{pq}^{\infty}$ belongs to $\mathfrak{C}(Q, \pi)$. Moreover, $H_nN$ has a series of submodules whose factors are isomorphic to the modules  $E_{pq}^{\infty}$ for $p+q=n$. The conclusion follows, then, from the closure of $\mathfrak{C}(Q, \pi)$ with respect to extensions.
\end{proof}

We now have everything in place to prove Theorem A. 
 
\begin{proof}[{\bf Proof of Theorem A}]

First of all, we observe that it will follow that these cohomology and homology groups are finite if we can establish that they are bounded. To see this, suppose that there is a positive integer $m$ such that $mH^n(G,A)=0$ and $mH_n(G,A)=0$ for some $n\geq 0$. Then the exact sequence $0\rightarrow  A\stackrel{\mu_m}{\rightarrow} A\rightarrow A/mA\rightarrow~0$, where $\mu_m:A\to A$ is induced by multiplication by $m$, gives rise to exact sequences $0\rightarrow H^n(G,A)\rightarrow H^n(G,A/mA)$ and $0\rightarrow 
H_n(G,A)\rightarrow H_n(G,A/mA)$. Since $G$ has finite rank and $A/mA$ is finite, $H^n(G,A/mA)$ and $H_n(G,A/mA)$ are finite. Hence the same holds for $H^n(G,A)$ and $H_n(G,A)$. 

We begin by proving the result under the assumption that $A$ is a rationally irreducible $\mathbb ZG$-module. There is a well-known theorem of A.~Mal'cev [{\bf 8}, 3.1.6] that states that every solvable irreducible linear group is virtually abelian. Hence $G/C_G(A)$ is virtually abelian. Let $F={\rm Fitt}(G)$ and $N=F \cap  C_G(A)$.  According to Proposition 1.2(iii), $F$ is nilpotent and $G/F$ virtually abelian.  Therefore, $G/N$, too, is virtually abelian. As a result, we can find a torsion-free subgroup $G_0$ of $G$ such that $[G:G_0]$ is finite and $Q_0=G_0/N_0$ is abelian, where $N_0=N\cap G_0$. Our plan is to show that $H^n(G_0,A)$ and $H_n(G_0,A)$ are bounded for each $n\geq 0$. An application of the transfer maps in cohomology and homology will then yield that 
$H^n(G,A)$ and $H_n(G,A)$ are bounded and thus finite.

First we deal with cohomology. Setting $\tilde{A }=(A\otimes \mathbb Q)/A$, we will employ the LHS cohomology spectral sequence for the group extension $1\to N_0\to G_0\to Q_0\to 1$ to study $H^n(G_0,\tilde{A})$. The universal coefficient theorem and Lemma 2.2 give rise to the following chain of isomorphisms for $p, \ q\geq 0$.
\begin{equation} H^p(Q_0, H^q(N_0,\tilde{A})) \cong H^p(Q_0 , {\rm Hom}(H_qN_0,\tilde{A}))\cong {\rm Ext}^p_{\mathbb ZQ_0}(H_qN_0,\tilde{A}).\end{equation}
Since $[G:G_0]$ is finite, any $\mathbb ZG_0$-submodule of $A$ that is $\pi$-spectral as an abelian group generates a $\mathbb ZG$-submodule with the same property. As a result, $A$ must not contain any nontrivial $\mathbb ZQ_0$-submodules that are $\pi$-spectral as abelian groups.
Consequently, the same is true for $A\otimes \mathbb Q$. 
Thus, by Lemmas 2.7 and 2.6, ${\rm Ext}^p_{\mathbb ZQ_0}(H_qN_0,\tilde{A})$ is bounded for $p, \ q\geq 0$. It follows, then, that $H^n(G_0,\tilde{A})$ is bounded for $n\geq 0$. Moreover, an identical argument establishes that $H^n(G_0,A\otimes \mathbb Q)$ is bounded for $n\geq 0$. Applying the long exact cohomology sequence, we infer that $H^n(G_0,A)$ is bounded for $n\geq 0$. 

We investigate the homology groups $H_n(G_0,A)$ directly, relying as above on an LHS spectral sequence. Dual to (2.2) are the isomorphisms 
\begin{equation*} H_p(Q_0, H_q(N_0,A)) \cong H_p(Q_0 , H_qN_0\otimes A)\cong {\rm Tor}_p^{\mathbb ZQ_0}(H_qN_0,A)\end{equation*}
for $p, \ q\geq 0$. Appealing again to Lemmas 2.7 and 2.6, we deduce that ${\rm Tor}_p^{\mathbb ZQ_0}(H_qN_0,A)$ is bounded for $p, \ q\geq 0$. Therefore, $H_n(G_0,A)$ is bounded for $n\geq 0$. This completes the proof for the case where $A$ is rationally irreducible.

Finally, we prove the general case of the theorem by induction on $h(A)$, the case $h(A)=0$ being trivial. Assume that $h(A)\geq 1$, and  let $B$ be a rationally irreducible submodule of $A$ such that the additive group of $A/B$ is torsion-free. The conclusion of the theorem will follow from the rationally irreducible case together with the inductive hypothesis if we manage to show that $A/B$ does not contain a nontrivial submodule that is $\pi$-spectral as an abelian group. To accomplish this, suppose that $A/B$ possesses such a submodule and take $C$ to be a submodule of $A$ properly containing $B$ such that $C/B$ is $\pi$-spectral. Now set $\Gamma=C\rtimes G$. By the case of the theorem proved above, $H^2(\Gamma/B,B)$ is finite. Applying Proposition~1.1, we obtain a $\pi$-spectral subgroup $X$ of $\Gamma$ such that  $[\Gamma:BX]<\infty$. Notice that $X\cap C$ is a nontrivial normal subgroup of $BX$. Thus the $\mathbb ZG$-submodule of $C$ generated by $X\cap C$ is $\pi$-spectral as an abelian group, which contradicts our hypothesis concerning $A$. We conclude, then, that $A/B$ does not possess any nontrivial submodules that are $\pi$-spectral as abelian groups. Hence Theorem~A has now been proven.

\end{proof}

\section{Near supplements}

In this section, we prove the principal group-theoretic result of the paper, Theorem B, and explore some of its consequences. In addition to Theorem A, the proof of Theorem B appeals to the following result of Robinson.

\begin{proposition}{\rm (Robinson [{\bf 11}, Theorem 2.1])} Let $G$ be a $\pi$-minimax group and $A$ a $\mathbb ZG$-module that is torsion and has finite rank qua abelian group. If the additive group of $A$ fails to have any quasicyclic $p$-subgroups for every $p\in \pi$, then $H^n(G,A)$ is torsion for $n\geq 0$.
\end{proposition}

Robinson's proposition immediately yields  

\begin{lemma} Let $G$ be a solvable group of finite rank. Suppose that  $T$ is a torsion normal subgroup of $G$ such that $G/T$ is $\pi$-minimax. Then $T$ has a $\pi$-minimax near supplement. 
\end{lemma}

\begin{proof} The proof is by induction on the length of the derived series of $T$. First suppose that $T$ is abelian. Let $\bar{\pi}={\rm spec}(G/T)$ and $S$ be the $\bar{\pi}$-torsion part of $T$. By Proposition 3.1, $H^2(G/T,T/S)$ is torsion. Invoking Proposition 1.1, we conclude that $G$ contains a  subgroup $X$ such that $S<X$, $X\cap T/S$ is finite, and $[G:TX]$ is finite. Since $X$ is clearly $\pi$-minimax, this establishes the base case. 

Assume that $T'\neq 1$. By the abelian case, $G$ has a subgroup $Y$ such that $[G:TY]$ is finite and $Y/T'$ is $\pi$-minimax. Applying the inductive hypothesis to $T'$ inside $Y$, we obtain a $\pi$-minimax subgroup $X<Y$ with $[Y:T'X]$ finite. It follows, then, that $[G:TX]$ is finite, thus completing the proof. 
\end{proof}

Having stated Lemma 3.2, we are ready to proceed with the proof of Theorem~B.

\begin{proof} [{\bf Proof of Theorem B}] To begin with, we treat the case where $K$ is torsion-free abelian and rationally irreducible as a $\mathbb ZG$-module. Within this case, we distinguish two subcases: (1) $K$ has no nontrivial $\pi$-minimax $G$-invariant subgroups; (2) $K$ has a nontrivial $\pi$-minimax $G$-invariant subgroup. 
First we dispose of case (1). In this situation, we can apply Theorem A to deduce that $H^2(G/K,K)$ is finite. According to Proposition 1.1, this means that there is a subgroup $X$ of $G$ such that $X\cap K=1$ and $[G:KX]$ is finite. Then $X$ is $\pi$-minimax, yielding the conclusion. Next we consider case (2).  Let $M$ be a nontrivial $G$-invariant $\pi$-minimax subgroup  of $K$. Because $K/M$ is torsion, Lemma 3.2 furnishes a subgroup $X$ containing $M$ such that $X/M$ is $\pi$-minimax and $[G:KX]<\infty$. As $X$ is plainly $\pi$-minimax, this concludes the argument for case (2).

Now we tackle the case where $K$ is virtually torsion-free. We proceed by induction on $h(K)$, the case $h(K)=0$ being trivial. Suppose that $h(K)\geq 1$. By Proposition 1.2(ii), $K$ has a characteristic subgroup $K_0$ of finite index that possesses a characteristic series of finite length in which each factor is a torsion-free abelian group. Hence we can find a $G$-invariant subgroup $L$ of $K_0$ such that $K_0/L$ is a torsion-free abelian group that is rationally irreducible as a $\mathbb ZG$-module. Consequently, by the case established in the first paragraph, there is a subgroup $Y$ such that $L<Y$, $Y/L$ is $\pi$-minimax, and $[G:K_0Y]<\infty$. Since $Y/L$ is virtually torsion-free and $h(L)<h(K)$, we can apply the inductive hypothesis to $L$ inside $Y$. This yields a $\pi$-minimax subgroup $X<Y$ such that $[Y:LX]$ is finite. It follows, then, that $[G:KX]$ is finite. Thus $X$ can serve as the desired subgroup.

Finally, we deal with the general case. Setting $T=\tau(K)$, we apply the virtually torsion-free case to $K/T$, thereby obtaining a subgroup $Y$ containing $T$ such that $Y/T$ is $\pi$-minimax and $[G:KY]$ is finite. By Lemma 3.2, $Y$ has a $\pi$-minimax subgroup $X$ such that $[Y:TX]$ is finite. Since $[G:KX]$ is thus finite, this completes the proof of Theorem B. 
\end{proof}

 \begin{remark}{\rm Although Theorem A holds for all groups of finite rank that are virtually torsion-free, Theorem B cannot be extended to the case where $G/K$ is not minimax. The reason is that Proposition 3.1 only applies to minimax groups (see [{\bf 11}, Section~6]).} 
 \end{remark}
 
 Below we describe an example that demonstrates that the hypothesis that $G/K$ is virtually torsion-free cannot be removed from Theorem B. 

\begin{example} {\rm Assume that $p$ and $q$ are distinct primes. Define an action of $C_\infty=\langle t\rangle$ on $\mathbb Z[1/pq]$ by
$t\cdot x=px$ for $x\in \mathbb Z[1/pq]$. Let $G=\mathbb Z[1/pq]\rtimes \langle t\rangle$ and $K=\mathbb Z[1/p]$. Then $K\lhd G$, and $G/K$ is $\{q\}$-minimax but not virtually torsion-free.  Moreover, $K$ has no $\{q\}$-minimax near supplement. To see this, suppose that  $X$ is a near supplement to $K$. Then $K\cap X\neq 1$; otherwise $X$ would contain a quasicyclic subgroup. Since $K\cap X\lhd KX$, it follows that $K\cap X$ contains a copy of $\mathbb Z[1/p]$. Thus $X$ cannot be $\{q\}$-minimax.}
\end{example}

As shown below in Corollary 3.4, Theorem B implies that, if $G/K$ is not virtually torsion-free, one is still assured of finding a $\pi$-minimax subgroup $X$ satisfying the weaker condition that  $h(KX)=h(G)$. We refer to such a subgroup as a {\it Hirsch-length supplement}. Furthermore, if $G/K$ happens to be finitely generated, then any Hirsch-length supplement is necessarily a near supplement. This is due to the fact that, in a finitely generated solvable minimax group, every subgroup with the same Hirsch length as the group must have finite index (Lemma 3.5). 

\begin{corollary} Let $G$ be a solvable group of finite rank. Assume that $K$ is a normal subgroup of $G$ such that $G/K$ is $\pi$-minimax. Then $K$ has a $\pi$-minimax Hirsch-length supplement. Furthermore, if $G/K$ is finitely generated, then any Hirsch-length supplement to $K$ is a near supplement. 
\end{corollary}

\begin{proof} Let $L$ be a normal subgroup of $G$ containing $K$ such that $L/K$ is the torsion radical of $G/K$. Applying Theorem B, we obtain a $\pi$-minimax subgroup $X$ of $G$ such that $[G:LX]<\infty$. Moreover, 
\begin{equation*} h(KX)=h(K)+h(X)-h(K\cap X)=h(L)+h(X)-h(L\cap X)=h(LX).\end{equation*}
Hence $X$ is a Hirsch-length supplement to $K$. 
It remains to show that, if $G/K$ is finitely generated, then every Hirsch-length supplement to $K$ is  really a near supplement. This assertion is an immediate consequence of Lemma 3.5 below. 
\end{proof}

The following property is undoubtedly well known; nonetheless, for the reader's convenience, we provide a proof. 

\begin{lemma} Let $G$ be a finitely generated solvable minimax group. If $H<G$ and $h(H)=h(G)$, then $[G:H]$ is finite. 
\end{lemma}

\begin{proof}

The proof is by induction on $m(G)$. If $m(G)=0$, then the conclusion follows immediately.  Assume that $m(G)>0$. Let $A$ be an infinite abelian subgroup with $m(A)$ as small as possible subject to the condition that $[G:N_G(A)]$ is finite.  Without any real loss of generality, we may replace $G$ by $N_{G}(A)$ and $H$ by ${H\cap N_{G}(A)}$, rendering $A$ normal in $G$.  In this case, we have 
\begin{equation*} h(AH/A) = h(H)-h(H\cap A)\geq h(G)-h(A) = h(G/A),\end{equation*}
implying that $h(AH/A)=h(G/A)$. Hence, by the inductive hypothesis, $AH$ has finite index in~$G$. 

At this juncture, we distinguish the case of $A$ being torsion from that where $A$ is not torsion. First we suppose that $A$ is torsion. Because $G$ is finitely generated, $AH$ is also finitely generated. Thus there is a finite subset $\mathcal{F}$ of $A$ such that $\langle H\cup \mathcal{F}\rangle=AH$. Since $A$ is a \v{C}ernikov group,  it has a characteristic finite subgroup $F$ containing $\mathcal{F}$. Then $FH=AH$, so that $[AH:H]$ is finite. Hence $[G:H]$ is finite.

Next we treat the case that  $A$ is not torsion. 
Since $H\cap A$ is normal in $AH$, our choice of $A$ ensures that $H\cap A$ is either finite or has finite index in $A$. 
If $H\cap A$ is finite, then $h(H)=h(G/A)<h(G)$, a contradiction.
Therefore, $H\cap A$ has finite index in $A$. It follows, then, that $[AH:H]$ is finite. Thus $[G:H]$ is finite.
\end{proof}

We conclude this section by applying Corollary 3.4 to derive a new result on the homological dimension of solvable groups. It is conjectured in \cite{peter3} that, if $k$ is a field and $G$  a solvable group, then $\mathrm{hd}_{k}(G)$ is either infinite or equal to the Hirsch length of $G$. This is established by U. Stammbach \cite{stammbach} if $k$ has characteristic zero and proved in  [{\bf 4}, Theorem I.3] for arbitrary fields in the abelian-by-polycyclic case. We refer the reader to \cite{peter3} for background information on this problem, as well as to \cite{bieri} for the basic facts about homological dimension. Here we employ Corollary~3.4 to generalize the argument in \cite{peter3} for the abelian-by-polycyclic case, obtaining the following proposition.  

\begin{proposition}
Let $p$ be a prime and $k$ a field of characteristic $p$. Suppose that $G$ is an extension of an abelian group by a solvable group that has no quasicyclic $p$-sections.  If $G$ has finite homological dimension over $k$, then $\mathrm{hd}_{k}(G)=h(G)$.
\end{proposition}

\begin{proof}

That $G$ has finite homological dimension over $k$ yields that $G$ has finite Hirsch length. In addition, it implies that $G$ has no $k$-torsion, meaning that any finite element order in $G$ is invertible in $k$.  Let $T=\tau(G)$.  Appealing to the LHS spectral sequence, we will argue that $\mathrm{hd}_k(G/T)=\mathrm{hd}_k(G)$. Since $\mathrm{hd}_k(T)=0$, we deduce right away that $\mathrm{hd}_k(G/T)\geq \mathrm{hd}_k(G)$. To verify the reverse inequality, set $n=\mathrm{hd}_k(G/T)$ and let $M$ be a $k[G/T]$-module such that $H_n(G/T,M)\neq 0$. As $H_n(G,M)=H_n(G/T,M)$, it follows  that $\mathrm{hd}_k(G)\geq n$. Because ${\mathrm{hd}_k(G/T)=\mathrm{hd}_k(G)}$, we may, without any real loss of generality, assume that $T=1$. As observed in [{\bf 3}, Proposition 6.14], there is a finitely generated subgroup of $G$ with the same homological dimension over $k$ as $G$.  In addition, there is a finitely generated subgroup with the same Hirsch length as $G$. This means that we can find a finitely generated subgroup enjoying both of these properties.  Consequently, it suffices to consider the case where $G$ is finitely generated and therefore, by the principal result in \cite{robinson4}, minimax. 

Invoking Corollary 3.4, and replacing $G$ by a subgroup of finite index if necessary, we may assume that $G$ has a $\{p\}'$-minimax subgroup $X$ and an abelian normal subgroup $A$ such that $AX=G$. The action of $X$ on $A$ by conjugation allows us to construct a semidirect product $A\rtimes X$. Moreover, there is an epimorphism
$\phi: A\rtimes X\to G$ such that $B={\rm Ker}\,\phi$ is isomorphic to $A\cap X$. We claim that, if $\mathrm{hd}_k(A\rtimes X)=h(A\rtimes X)$, then $\mathrm{hd}_k(G)=h(G)$.  To establish this claim, assume that
${\mathrm{hd}_k(A\rtimes X)}={h(A\rtimes X)}$. From the LHS spectral sequence we know that
${\mathrm{hd}_k(A\rtimes X)\leq} {\mathrm{hd}_k(B) + \mathrm{hd}_k(G)}$. Hence $\mathrm{hd}_k(G)\geq h(A\rtimes X)-h(B)=h(G).$ Also, $\mathrm{hd}_k(G)\leq h(G)$ by [{\bf 3}, Theorem 7.11].
Therefore, the above claim holds; in other words, we do not really lose any generality in supposing that $G=A\rtimes X$. 

The argument advanced by U. Stammbach \cite{stammbach} for fields of characteristic zero also serves to show that $\mathrm{hd}_{k}(X)=h(X)$ (see [{\bf 4}, Lemma I.9]). Let $m$ be the product of all the primes in ${\rm spec}(A)$. Then $\tilde A=A\otimes \mathbb Z[1/m]$ is isomorphic to the direct sum of $r$ copies of $\mathbb Z[1/m]$, where $r=h(A)$.   Let $\theta$ be the automorphism of $\tilde A$ defined by $\theta(x)=mx$. Set $K=\tilde A\rtimes\langle\theta\rangle$. The action of $X$ on $A$ induces an action of $X$ on $\widetilde A$. Now define an action of $X$ on $K$ by employing the action of $X$ on $\tilde A$ and allowing $X$ to centralize $\theta$. Using this action, form the semidirect product $\tilde G=K\rtimes X$. As an ascending HNN extension of the direct sum of $r$ copies of $\mathbb Z$, $K$ is constructible and thus an inverse duality group (see [{\bf 2}, Theorem~9]).  By [{\bf 3}, Theorem 9.4], it follows  that $K$ has type FP and $H^n(K,kK)$ is $k$-free. Hence, in view of  [{\bf 3}, Theorem~5.5], $\mathrm{hd}_{k}(\widetilde{G})=\mathrm{hd}_k(K)+\mathrm{hd}_k(X)$. Moreover, according to  [{\bf 4}, Proposition~I.11], $\mathrm{hd}_k(K)=h(K)$. In addition, since $\tilde{G}$ can be viewed as an ascending HNN extension of $G$, the Mayer-Vietoris sequence yields that ${\mathrm{hd}_k(\tilde{G})\leq \mathrm{hd}_k(G)+1}$. 
We may thus argue as follows. 
\begin{equation*} \mathrm{hd}_{k}(G)\ge\mathrm{hd}_{k}(\tilde G)-1=h(K)+h(X)-1=h(A)+h(X)=h(G).\end{equation*} 
As remarked above, the reverse inequality is already known to hold; hence $\mathrm{hd}_k(G)=h(G)$. 

\end{proof}

\section{Near complements}

In this section, we consider the same situation as in Theorem B; that is, we suppose that $G$ is a solvable group of finite rank and $K\lhd G$ such that $G/K$ is $\pi$-minimax and virtually torsion-free. Here our goal is to determine circumstances under which Theorem B can be strengthened to yield a near complement, rather than just a $\pi$-minimax near supplement. In order to ensure this, it will be necessary to impose two quite stringent conditions on $K$. First, we will assume that $K$ is Noetherian as a $G$-operator group; second, we will require that $K$ is, in some sense, the ``antithesis" of a torsion-free $\pi$-minimax group. 

Before we concern ourselves with the second hypothesis, we point out that the Noetherian property, as well as the condition that the quotient is virtually torsion-free, applies to any normal subgroup in a solvable minimax group that satisfies the maximum condition on normal subgroups. Moreover, solvable minimax groups with the latter attribute are ubiquitous, especially in geometric group theory. In particular, every constructible solvable group enjoys this property. Recall that the class of constructible groups is the smallest class containing the trivial group that is closed under forming finite extensions, generalized free products in which both factors as well as the amalgamated subgroup are constructible, and HNN extensions in which the base group and associated subgroups are constructible. Furthermore, as shown in \cite{constructible}, the class of constructible solvable groups is the smallest class of groups containing the trivial group that is closed with respect to forming extensions by finite solvable groups and ascending HNN extensions with base group in the class. As well as being of interest to geometers, constructible solvable groups are important in homological algebra since the torsion-free constructible groups are precisely those solvable groups that have type ${\rm FP}$ (see \cite{peter2}). 

Now we turn our attention to the second hypothesis that we will require for our near-complement result, namely, that $K$ fails to be a torsion-free $\pi$-minimax group in some drastic fashion. In order to impart a more precise form to this idea, we introduce the following class. 

\begin{definition}{\rm 
The class  $\mathfrak{X}_{\pi}$ is the class of all groups $G$ such that every $\pi$-minimax quotient of $G$ is torsion. }
\end{definition}

To illustrate this definition, we consider some examples. Notice that every torsion group belongs to $\mathfrak{X}_{\pi}$, whereas all $\pi$-minimax groups that are not torsion fall outside $\mathfrak{X}_{\pi}$.
Another elementary observation is that a torsion-free abelian group of rank one is a member of 
$\mathfrak{X}_{\pi}$ if and only if it is not $\pi$-minimax.  From this example we can see that the class $\mathfrak{X}_{\pi}$ fails to be closed under passage to subgroups. 

Although $\mathfrak{X}_{\pi}$ is not subgroup-closed, the class enjoys the four closure properties described below in Lemma 4.1. Since the proofs of these are straightforward, we leave them to the reader. 

\begin{lemma}  (i) If $G$ belongs to $\mathfrak{X}_{\pi}$, then so does every subgroup of finite index in $G$.   

\noindent (ii) The class $\mathfrak{X}_{\pi}$ is closed under forming quotients and extensions. 

\noindent (iii)  Let $G$ be a group. If $\{K_\alpha:\alpha\in I\}$ is a family of normal subgroups of $G$ such that each $K_\alpha$ lies in $\mathfrak{X}_\pi$, then the join of the $K_\alpha$  also belongs to $\mathfrak{X}_\pi$. 
\end{lemma}

Statements (ii) and (iii) in Lemma 4.1 make $\mathfrak{X}_{\pi}$ a radical class in the sense used in [{\bf 9}, Section 1.3]. 
Employing the terminology from [{\bf 9}, Section 1.3], we 
define the {\it $\mathfrak{X}_{\pi}$-radical} of a group $G$, denoted $\rho_{\pi}(G)$, to be the join of all the normal  $\mathfrak{X}_{\pi}$-subgroups of $G$. In studying the properties of $\mathfrak{X}_{\pi}$ and the $\mathfrak{X}_{\pi}$-radical, we will draw upon the wealth of information about radical classes contained in [{\bf 9}, Section 1.3]. 

Our definition of the class $\mathfrak{X}_{\pi}$ is inspired by the notion of an {\it upper-finite group} from [{\bf 10}, Section 10.4]; this is a group whose finitely generated quotients are all finite. To grasp the connection, consider the special case $\pi=\emptyset$. The class $\mathfrak{X}_{\emptyset}$ is the class of groups all of whose polycyclic quotients are finite. Using the fact that solvable groups of finite rank are torsion-by-nilpotent-by-polycyclic (Proposition~1.2(iii)), it follows that a solvable group of finite rank is an $\mathfrak{X}_{\emptyset}$-group if and only if it is upper-finite. Hence the 
$\mathfrak{X}_{\emptyset}$-radical of a solvable group of finite rank coincides with the upper-finite radical from [{\bf 10}, Section 10.4]. 

The result on near complements that we will prove in this section is stated below.

\begin{thmC} Let $G$ be a solvable group of finite rank and $K\lhd G$. Assume that $K$ is a member of $\mathfrak{X}_{\pi}$ and $K$ is Noetherian as a $G$-operator group. Suppose further that $G/K$ is $\pi$-minimax and virtually torsion-free. Then there is a near complement to $K$ in $G$.
\end{thmC}

The first step towards proving Theorem C is to establish that the property of belonging to $\mathfrak{X}_{\pi}$ is inherited by submodules of certain Noetherian modules that lie in $\mathfrak{X}_{\pi}$. For this purpose, we require the following elementary observation.

 \begin{lemma} Assume that $G$ is a group and $K$ a normal subgroup of $G$ that is solvable and has finite rank. Then $K$ is Noetherian as a $G$-operator group if and only if every $G$-operator quotient of $K$ is virtually torsion-free.  
 \end{lemma}
 
\begin{proof}  For the ``only if" part of the lemma, we need only observe that a Noetherian $G$-operator group that is torsion and solvable of finite rank must be finite. To verify the ``if" direction, suppose that $K$ is not Noetherian. Then $K$ has an infinite ascending series of $G$-invariant subgroups that never stabilizes. Since the sequence of Hirsch lengths of the subgroups in this series must eventually become constant, the series gives rise to a $G$-operator quotient of $K$ with an infinite torsion radical. 
\end{proof}

Below we prove our inheritance property for modules.

\begin{lemma}
Let $G$ be an abelian group and $A$ a Noetherian $\mathbb ZG$-module that has finite rank as an abelian group. If the underlying abelian group of $A$ is in the class $\mathfrak{X}_{\pi}$, then the same is true for every submodule of $A$. 
\end{lemma}

\begin{proof} We argue by induction on $h(A)$. For $h(A)=0$, the conclusion is trivially true. Assume that $h(A)\geq 1$. 
First we show that every rationally irreducible submodule of $A$ lies in $\mathfrak{X}_{\pi}$. Suppose that there is a rationally irreducible submodule $B$ outside the class $\mathfrak{X}_{\pi}$.  Set $\bar{A}=A/B$.  By Lemma 4.2, the additive groups of both $\bar{A}$ and $B$ are virtually torsion-free. Also, in view of the inductive hypothesis, every $\mathbb ZG$-submodule of $\bar{A}$ belongs to $\mathfrak{X}_{\pi}$. On the other hand, $B$ has no infinite $\mathbb ZG$-submodules that lie in $\mathfrak{X}_{\pi}$. Therefore, by Proposition A, ${\rm Ext}_{\mathbb ZG}^1(\bar{A}, B)$ is bounded. According to Proposition 2.1, this means that $A$ has a submodule $A_0$ of finite index such that $B/\tau(B)$ is a homomorphic image of $A_0$.  The closure properties of $\mathfrak{X}_{\pi}$ imply, then, that $B$ is a member of $\mathfrak{X}_{\pi}$, yielding a contradiction. Therefore, every rationally irreducible submodule of $A$ belongs to $\mathfrak{X}_{\pi}$. 

To complete the proof, we let $B$ be an arbitrary $\mathbb ZG$-submodule of $A$. If $B$ is finite, then $B$ is immediately seen to belong to $\mathfrak{X}_\pi$; hence we assume that  $B$ is infinite. Let $C$ be a rationally irreducible submodule of $B$. We have that $C$ belongs to $\mathfrak{X}_{\pi}$ by what was proved above. Moreover,  $B/C$ lies in 
$\mathfrak{X}_{\pi}$ by virtue of the inductive hypothesis. Thus $B$ is a member of $\mathfrak{X}_{\pi}$.

\end{proof}

We wish to discern inheritance properties similar to Lemma 4.3 for normal subgroups of solvable groups of finite rank.  In order to accomplish this, we require the notion of a nilpotent action. 

\begin{definition} {\rm Assume that $G$ is a group and $N$ a $G$-operator group. We
define the {\it lower central
$G$-series}
\[ \dots < \gamma^G_3N <  \gamma^G_2N <
\gamma^G_1N\] of $N$ as follows: $\gamma^G_1N=N$;
$\gamma^G_iN=\langle a(g\cdot b)a^{-1}b^{-1}\ | \ a\in N, b\in
\gamma^G_{i-1}N, g\in G\rangle$ for
$i>1$.

We say that the action
of
$G$ on
$N$ is {\it nilpotent} if
there is a nonnegative integer
$c$ such that $\gamma^G_{c+1}N=1$. The smallest
such integer $c$ is called the {\it nilpotency class} of the action.}
\end{definition}

In studying nilpotent actions, the following well-known property is exceedingly useful. 

\begin{proposition} Assume that $G$ is a group and $N$ a $G$-operator group. Then, for each $i\geq 1$, there is an epimorphism
\begin{equation*} \theta_i: \underbrace{G_{\rm ab}\otimes \dots \otimes G_{\rm ab}}_{i-1}\otimes (N/\gamma^G_2N)\rightarrow \gamma^G_iN/\gamma^G_{i+1}N.\end{equation*}
\end{proposition}

The above epimorphism can be employed to prove the lemma below.

\begin{lemma} Let $G$ be a group such that $G_{\rm ab}$ has finite torsion-free rank. Assume that $N$ is a $G$-operator group upon which $G$ acts nilpotently. 
If $N$ belongs to $\mathfrak{X}_{\pi}$, then $\gamma^G_i(N)$ is in $\mathfrak{X}_{\pi}$ for all $i\geq 1$. 
\end{lemma} 

\begin{proof} The conclusion follows immediately from Proposition 4.4 and Lemma 4.6 below.
\end{proof}

\begin{lemma}
Let $A$ and $B$ be abelian groups. If $A$ belongs to $\mathfrak{X}_{\pi}$ and $B$ has finite torsion-free rank, then $A\otimes B$ is a member of  $\mathfrak{X}_{\pi}$.
\end{lemma}

\begin{proof}
Choose a free abelian subgroup $C$ of $B$ such that $B/C$ is torsion. Then $C\cong \mathbb Z^{n}$, where $n=h(B)$, and so $A\otimes C\cong \underbrace{A\oplus\dots\oplus A}_{n}$. The class $\mathfrak{X}_{\pi}$  is closed under extensions and quotients; hence every quotient of $A\otimes C$ is in $\mathfrak{X}_{\pi}$ . Also, as a torsion group, $A\otimes (B/C)$ belongs to  $\mathfrak{X}_{\pi}$ . It follows, then, from the exact sequence 
$A\otimes C\to A\otimes B\to A\otimes (B/C)\to 0$
that $A\otimes B$ lies in $\mathfrak{X}_{\pi}$. 
\end{proof}

Lemma 4.5 allows us to establish the following two inheritance properties for solvable $\mathfrak{X}_{\pi}$-groups of finite rank. The second of these, Lemma 4.8, will play an essential role in the proof of Theorem C.  

\begin{lemma} Let $M$, $N$, and $P$ be nilpotent normal subgroups of a group $G$ such that
$M<N<P$ and the following three conditions hold.
\vspace{5pt}
 
 (i) The quotient $G/P$ is abelian.
 
 (ii) The group $P$ has finite rank. 
 
 (iii) The subgroup $N$ is Noetherian as a $G$-operator group. 
\vspace{5pt}

\noindent If $N$ belongs to $\mathfrak{X}_{\pi}$, then $M$ does too. 
\end{lemma}

\begin{proof} 
Our proof is by induction on the nilpotency class of the action of $P$ on $N$. If $P$ acts trivially, then we can deduce the conclusion from Lemma 4.3, regarding $N$ and $M$ as $G/P$-modules. Assume that the nilpotency class $c$ of the action exceeds one, and let $A=\gamma_c^P(N)$.  
Then $A$ is a $\mathfrak{X}_{\pi}$-group by Lemma 4.5. 
Consider the chain
\begin{equation*} M/M\cap A<N/A<P/A,\end{equation*}
the second term of which is in $\mathfrak{X}_{\pi}$. The action of $P/A$ on $N/A$ is nilpotent of class $c-1$. Consequently, the inductive hypothesis yields that $M/M\cap A$  is in $\mathfrak{X}_{\pi}$.  
Next we look at the groups 
$M\cap A$ and $A.$ 
Treating these groups as $G/P$-modules, it follows from Lemma 4.3 that $M\cap A$ belongs to $\mathfrak{X}_{\pi}$. Therefore, $M$ is in $\mathfrak{X}_{\pi}$, as desired.

\end{proof}

\begin{lemma} Let $G$ be a solvable group of finite rank such that $\tau(G)$ is a \v{C}ernikov group. Suppose that $M$ and $N$ are normal subgroups of $G$ such that $M < N$ and $N$ is Noetherian as a $G$-operator group.  If $N$ is in $\mathfrak{X}_{\pi}$, then so is $M$. 
\end{lemma}

\begin{proof}  

By Proposition 1.2(iii), $N$ can be expressed as an extension of a nilpotent group by a polycyclic one. Since $N$ belongs to the class $\mathfrak{X}_{\pi}$, this means that $N$ must be virtually nilpotent. By passing to a $G$-invariant subgroup of finite index, we can assume that $N$ is nilpotent. In addition, $F={\rm Fitt}(G)$ is nilpotent, and $G/F$ is virtually abelian. The latter assertion implies that $G$ has a subgroup $G_0$ of finite index such that $F<G_0$ and $G_0/F$ is abelian. J. S. Wilson shows in [{\bf 13}, Theorem~A] that any normal subgroup of a group $\Gamma$ that is Noetherian as a $\Gamma$-operator group is also Noetherian with respect to any subgroup of finite index in $\Gamma$. Hence $N$ is Noetherian as a $G_0$-operator group. Therefore, the conclusion follows immediately by applying Lemma 4.7 to the chain 
$M<N<F$ inside the group~$G_0$. 

\end{proof}

Armed with the above lemma, we can proceed with the proof of Theorem C. 

\begin{proof}[{\bf Proof of Theorem C}] 

We argue by induction on the length of the derived series of $K$. First suppose that $K$ is abelian. Appealing to Lemma 4.8 and Theorem~A, we deduce that $H^2(G/K,K)$ is finite. Thus $G$ has a subgroup $Y$ such that $K\cap Y$ is finite and $[G:KY]<\infty$. The former property implies that $Y$ is residually finite. As a result, $Y$ has a subgroup $X$ of finite index with $X\cap K=1$. The subgroup $X$, then, can serve as the near complement that we seek.

Next assume that the derived length of $K$ is greater than one. By the abelian case, $G$ contains a subgroup $Y$ such that $Y\cap K=K'$ and $[G:KY]<\infty$. According to Lemma 4.8, $K'$ belongs to $\mathfrak{X}_{\pi}$. Hence we can apply the inductive hypothesis to $K'$ inside of $Y$ to obtain $X<Y$ such that $K'\cap X=1$ and $[Y:K'X]$ is finite. Then $X$ fulfills our requirements. 
\end{proof} 

In view of the observations made at the beginning of the second paragraph of this section, Theorem C has the following important special case.

\begin{corollary} Let $G$ be a solvable minimax group that satisfies the maximum condition on normal subgroups. Assume that $K$ is a normal $\mathfrak{X}_{\pi}$-subgroup of $G$ such that $G/K$ is $\pi$-minimax. Then $K$ has a near complement in $G$.
\end{corollary}

In addition to the assumption that $G/K$ is $\pi$-minimax, Theorem C requires three hypotheses: (1) $G/K$ is virtually torsion-free; (2) $K$ is Noetherian; (3) $K$ belongs to $\mathfrak{X}_{\pi}$. It is very easy to demonstrate that each of  these three conditions is indispensable. To accomplish this, we present three examples in which no near complement is present; in each, one of the three hypotheses is violated, while the other two hold. 

\begin{example} {\rm Let $p$ be a prime and $\Gamma$ the group consisting of all matrices of the form
\begin{equation*} \begin{pmatrix}
1&*&*\\ 0&\dagger&*\\ 0&0&1
\end{pmatrix},\end{equation*} 
where the entries $*$ above the diagonal are chosen from the ring $\mathbb Z[1/p]$ and the diagonal entry $\dagger$ is an integer power of $p$. Let $A$ be the central subgroup generated by
\begin{equation*} \begin{pmatrix}
1&0&1\\ 0&1&0\\ 0&0&1
\end{pmatrix}.\end{equation*} 
Let $G=\Gamma /A$ and $K$ be the central subgroup of $G$ generated by the image of 
\begin{equation*}\begin{pmatrix}
1&0&p^{-1}\\ 0&1&0\\ 0&0&1
\end{pmatrix}.\end{equation*} 
The group $G$ is a finitely generated solvable minimax group and $K\cong C_p$. 
If $\pi=\{p\}$, then $G/K$ is $\pi$-minimax, yet $K$ lacks a near complement. Notice that $K$ satisfies conditions (2) and (3), but $G/K$ fails to fulfill (1).
} 

\end{example}

\begin{example} {\rm Let $p$ be a prime and $G$ the nilpotent minimax group consisting of all matrices of the form
\begin{equation*} \begin{pmatrix}
1&*&\dagger\\ 0&1&*\\ 0&0&1
\end{pmatrix},\end{equation*} 
where the entries $*$  are integers and the entry $\dagger$ is from the ring $\mathbb Z[1/p]$.  Let $K=Z(G)$; that is, $K$ consists of all the matrices in $G$ whose $*$ entries are $0$.  For $\pi=\emptyset$, $G/K$ is $\pi$-minimax, and the extension $1\rightarrow K\rightarrow G\rightarrow G/K\rightarrow 1$ satisfies conditions (1) and (3), but not (2). In addition, there is no near complement to $K$ in $G$. }
\end{example}

\begin{remark} {\rm The group $G$ in Example 4.11 is not finitely generated. It would be interesting to discover whether there is a finitely generated example with the same characteristics.}
\end{remark}

\begin{example} {\rm Let $G$ be the group of 3x3 upper unitriangular matrices with integer entries. Set $K=Z(G)$. For $\pi=\emptyset$, $G/K$ is $\pi$-minimax; also, the extension $1\rightarrow K\rightarrow G\rightarrow G/K\rightarrow 1$ satisfies conditions (1) and (2), but not (3). Finally, we observe that $K$ does not have a near complement.}
\end{example}

One particular normal subgroup $K$ of a solvable group $G$ with finite rank that often satisfies the hypotheses of Theorem C is the subgroup $\rho_{\pi}(G)$. In order to show this, we require the following observation. 

\begin{proposition} Let $G$ be a virtually torsion-free solvable group of finite rank. If every subnormal abelian subgroup of $G$ is $\pi$-minimax, then $G$ must be $\pi$-minimax.
\end{proposition}

\begin{proof} We will invoke R. Baer's result [{\bf 1}, Hilfssatz 7.12] that a metabelian group with finite torsion radical is $\pi$-minimax if every normal abelian subgroup is $\pi$-minimax. To deduce the proposition from Baer's assertion, we argue by induction on $h(G)$, the case $h(G)=0$ being trivial. Assume that $h(G)>0$. By Proposition~1.2(ii), $G$ has an infinite abelian characteristic subgroup $A$ such that $G/A$ is virtually torsion-free. Let $B$ be a subnormal subgroup of $G$ such that $A< B$ and $B/A$ is abelian. It follows from Baer's result that $B$ is $\pi$-minimax. Therefore, by the inductive hypothesis, $G/A$ is $\pi$-minimax. Thus $G$ is $\pi$-minimax. 
\end{proof}

Proposition 4.13 allows us to establish the following lemma and its corollary. 

\begin{lemma}
Let $G$ be a solvable group with finite torsion-free rank. If $G$ has no nontrivial subnormal $\mathfrak{X}_{\pi}$-subgroups, then $G$ is $\pi$-minimax.
\end{lemma}

\begin{proof}
Suppose that $G$ is not $\pi$-minimax. We will show that $G$ must contain a nontrivial subnormal $\mathfrak{X}_{\pi}$-subgroup, proving the lemma. If $\tau(G)\neq 1$, then it can serve as the desired subgroup. Assume that $\tau(G)=1$. 
By Proposition 4.13, $G$ must have a subnormal abelian subgroup $A$ that is not $\pi$-minimax. Let $B$ be a subgroup of $A$ with the smallest possible Hirsch length such that $A/B$ is $\pi$-minimax. Then $B$ is a nontrivial group in $\mathfrak{X}_{\pi}$, and $B$ is subnormal in $G$.
\end{proof}

\begin{corollary} If $G$ is a solvable group with finite torsion-free rank, then $G/\rho_{\pi}(G)$ is a virtually torsion-free $\pi$-minimax group.
\end{corollary}

\begin{proof} The $\mathfrak{X}_{\pi}$-radical of $G/\rho_{\pi}(G)$ is trivial. Hence  $G/\rho_{\pi}(G)$ has no nontrivial subnormal $\mathfrak{X}_{\pi}$-subgroups. Therefore, it is $\pi$-minimax by Lemma 4.14. In addition, its torsion radical is trivial, making it virtually torsion-free.
\end{proof}

In light of Corollary 4.15, we can state the following corollary to Theorem C.

\begin{corollary} Let $G$ be a solvable group of finite rank such that $\rho_{\pi}(G)$ is Noetherian as a $G$-operator group. Then $\rho_{\pi}(G)$ has a near complement in $G$.
\end{corollary}

In passing, we remark that, if the hypothesis that $\rho_{\pi}(G)$ is Noetherian as a $G$-operator group is dropped from Corollary 4.16, we can still find a near supplement to  $\rho_{\pi}(G)$. This is a corollary to Theorem B. 

\begin{corollary} Let $G$ be a solvable group of finite rank. Then $\rho_{\pi}(G)$ has a $\pi$-minimax near supplement in $G$.
\end{corollary}

We conclude this section by observing that Corollary 4.16, and {\it a fortiori} Theorem C, cannot be strengthened to deliver a full complement to $K$. 

\begin{example}{\rm  Assume that $p$ is an odd prime. Let $Q$ be a free abelian group of rank two.
Put $A=\mathbb Z[1/p]$ and define an action of $Q$ on $A$ by having both generators act via multiplication by $p$. This makes $A$ into a Noetherian $\mathbb Z Q$-module. We have $H^2(Q,A)\cong H_0(Q,A)$ by Poincar\'e
duality. Hence, since $p\neq 2$, $H^2(Q,A)\neq 0$.
Therefore, there is at least one nonsplit extension of $A$ by $Q$.
In this extension, $A$ is the $\mathfrak{X}_\emptyset$-radical and fails to have a full complement.}
\end{example}

\begin{acknowledgement}{\rm The authors are greatly indebted to Derek Robinson for providing the argument for Proposition~A that appears in this version of the article. Our original proof of the proposition was considerably longer and yielded a result that was not as robust as the current form.  Moreover, the improvement to Proposition A allowed us to strengthen Theorem A in two respects: in our original formulation the conclusion was merely that the cohomology groups were torsion; in addition, the group $G$ had to be minimax.  Hence, not only did Derek Robinson's contribution make the paper more succinct and perspicuous, but it also served to enhance the main result.} 
\end{acknowledgement}


\begin{thebibliography}{12}


\bibitem[{\bf 1}]{baer}{\sc R. Baer.} Polyminimaxgruppen.  {\it Math. Ann. } {\bf 175} (1968), 1-43.

\bibitem[{\bf 2}]{constructible}{\sc R. Bieri} and {\sc G. Baumslag.} Constructable [sic] solvable groups. {\it Math. Z.} {\bf 151} (1976), no. 3, 249-257.

\bibitem[{\bf 3}]{bieri}{\sc R. Bieri.} {\it Homological Dimension of
Discrete Groups.} Queen Mary College Mathematics Notes (Mathematics
Department, Queen Mary College, London, 1976).

\bibitem[{\bf 4}]{peter3}{\sc M. R. Bridson} and {\sc P. H. Kropholler}. Dimension of elementary amenable groups. To appear in {\it J. Reine Angew. Math.}.
  
\bibitem[{\bf 5}]{hilton}{\sc P. J. Hilton} and {\sc U. Stammbach}. {\it A Course in Homological Algebra}, Second Edition (Springer, 1997). 

\bibitem[{\bf 6}]{peter2}{\sc P. H. Kropholler.}  Cohomological dimension of soluble groups. {\it J. Pure Appl. Algebra} {\bf 43} (1986), 281-287.

\bibitem[{\bf 7}]{robinson-lennox2}{\sc J. C. Lennox} and {\sc D. J. S. Robinson}. Soluble products of nilpotent groups. {\it Rend. Sem. Mat. Univ. Padova}~{\bf 62} (1980), 261-80. 

\bibitem[{\bf 8}]{robinson-lennox}{\sc J. C. Lennox} and {\sc D. J. S. Robinson}. {\it The Theory of Infinite Soluble Groups} (Oxford, 2004).

\bibitem[{\bf 9}]{robinson1}{\sc D. J. S. Robinson.}  {\it Finiteness Conditions and Generalized Soluble Groups, Part~1.} (Springer, 1972). 

\bibitem[{\bf 10}]{robinson2}{\sc D. J. S.  Robinson.}  {\it Finiteness Conditions and Generalized Soluble Groups, Part 2. } (Springer, 1972).  

\bibitem[{\bf 11}]{robinson4}{\sc D. J. S. Robinson.} On the cohomology of soluble groups of finite rank. {\it J. Pure Appl. Algebra}~{\bf 43} (1978), 281-287. 

\bibitem[{\bf 12}]{stammbach}{\sc U. Stammbach}. On the weak homological dimension of the group algebra of solvable groups. {\it J. London Math. Soc. (2)} {\bf 2} (1970), 567-570.

\bibitem[{\bf 13}]{wilson}{\sc J. S. Wilson}. Some properties of groups inherited by normal subgroups of finite index. {\it Math. Z.} {\bf 114} (1970), 19-21. 

\end{thebibliography}
\end{document}